\documentclass[reqno, 11pt]{amsart}

\usepackage[utf8]{inputenc}
\usepackage[T1]{fontenc}
\usepackage{amsfonts,amssymb,enumerate,amsthm,amsmath,color,graphicx}
\usepackage{hyperref}
\usepackage[all]{xy}

\usepackage[a4paper,margin=4cm,includehead,includefoot,headsep=12pt,footskip=18pt]{geometry}

\newtheorem{theor}{Theorem}[section]
\newtheorem{prop}[theor]{Proposition}
\newtheorem{defin}[theor]{Definition}
\newtheorem{lem}[theor]{Lemma}

\newtheorem{example}[theor]{Example}
\newtheorem{remark}[theor]{Remark}

\newcommand{\B}{\mathbb{B}}
\newcommand{\R}{\mathbb{R}}

\newcommand{\C}{\mathbb{C}}

\newcommand{\Q}{\mathbb{Q}}

\newcommand{\Z}{\mathbb{Z}}

\newcommand{\Vol}{\operatorname {Vol}}

\newcommand{\K}{K\"ahler}

\newcommand{\Ric}{\operatorname {Ric}}

\newcommand{\D}{D^{M_{\Omega,\mu}}}

\begin{document}

\title[On the Bergman metric of Cartan--Hartogs domains]{On the Bergman metric of Cartan--Hartogs domains}

\author{Andrea Loi}
\address{(Andrea Loi) Dipartimento di Matematica \\
         Universit\`a di Cagliari (Italy)}
         \email{loi@unica.it}

\author{Roberto Mossa}
\address{(Roberto Mossa) Dipartimento di Matematica \\
         Universit\`a di Cagliari (Italy)}
        \email{roberto.mossa@unica.it}
        
        \author{Fabio Zuddas}
\address{(Fabio Zuddas) Dipartimento di Matematica \\
         Universit\`a di Cagliari (Italy)}
        \email{fabio.zuddas@unica.it}

\thanks{
The authors are supported by INdAM and  GNSAGA - Gruppo Nazionale per le Strutture Algebriche, Geometriche e le loro Applicazioni, by GOACT and MAPS- Funded by Fondazione di Sardegna, by the Italian Ministry of Education, University and Research through the PRIN 2022 project "Developing Kleene Logics and their Applications" (DeKLA), project code: 2022SM4XC8  and partially funded by PNRR e.INS Ecosystem of Innovation for Next Generation Sardinia (CUP F53C22000430001, codice MUR ECS00000038).}

\subjclass[2000]{53C55, 32Q15, 53C24, 53C42 .} 
\keywords{Bergman metric; \K\ metrics; Cartan-Hartogs domains;  \K-Einstein metrics; \K-Ricci solitons; homogeneous \K\ manifold;  Bergman dual}

\begin{abstract}
We investigate the Bergman metric and the \emph{Bergman dual} on Cartan–Hartogs (CH) domains. 
For a bounded domain $D\subset\C^n$ with Bergman kernel $K_D(z,\bar z)$, we define the Bergman dual of $(D,g_D)$ as the pair $(D^*,g_D^*)$ where $D^*\subset\C^n$ is the maximal domain on which $K_D^*(z,\bar z):=K_D(z,-\bar z)$ is positive and $g_D^*$ is the Kähler metric with Kähler form $\omega_D^*:=-\tfrac{i}{2}\partial\bar\partial\log K_D^*$. 
We prove that, for a Cartan–Hartogs domain $M_{\Omega,\mu}$, the following conditions are equivalent: 
(i) $M_{\Omega,\mu}$ is biholomorphic to the unit ball, 
(ii) its Bergman metric is a Kähler–Ricci soliton, and 
(iii) its Bergman dual is finitely projectively induced up to rescaling. 
Conditions (i) and (ii) should be viewed as natural analogues of classical rigidity results: Yau’s problem and Cheng’s conjecture in the Kähler–Einstein setting, as well as Sha’s recent theorem on Kähler–Ricci solitons~\cite{Sha2024}. 
By contrast, condition (iii) emphasizes the new notion of Bergman dual, inspired by the duality phenomenon for bounded symmetric domains and their compact duals. 
We also compare our approach with other canonical metrics on CH domains, namely the metrics $g_{\Omega,\mu}$ and $\hat g_{\Omega,\mu}$ recently considered in~\cite{LMZdual2024}, and discuss open problems concerning the maximal domain of definition of the Bergman dual.
\end{abstract}
 
\maketitle

\tableofcontents

\section{Introduction}
The study of the Bergman metric of a bounded domain $D\subset \C^n$ 
is a central theme in several complex variables and Kähler geometry. 
Recall that the \emph{Bergman metric} $g_D$ is the Kähler metric on $D$ whose associated form is
\[
\omega_D \;=\; \frac{i}{2}\,\partial\bar\partial \log K_D,
\]
where $K_D(z,\bar z)$ denotes the Bergman kernel of $D$ restricted to the diagonal.  

The prototypical example is the unit ball
\[
\B^n \;=\; \{ z \in \C^n : |z|^2 < 1 \}, 
\qquad |z|^2 \;=\; |z_1|^2 + \cdots + |z_n|^2,
\]
for which
\[
K_{\B^n}(z,\bar z) = \frac{n!}{\pi^n}\,\frac{1}{(1- |z|^2)^{\,n+1}},
\qquad
\omega_{\B^n} =- (n+1)\,\frac{i}{2}\,\partial\bar\partial \log(1-|z|^2).
\]
Thus, up to the constant factor $(n+1)$, the Bergman metric of $\B^n$ coincides with the hyperbolic metric of constant negative holomorphic sectional curvature.  

A classical problem  is to understand which curvature conditions on $g_D$ force $D$ to be biholomorphic to $\B^n$.  
Among the natural conditions, perhaps the most immediate is the requirement that the Bergman metric has \emph{constant holomorphic sectional curvature}.  
This assumption is extremely restrictive and leads to strong rigidity phenomena.  
Qi-Keng Lu~\cite{Lu66} proved that if $g_D$ on $D \subset \C^n$ 
is complete with constant holomorphic sectional curvature, then 
$D$ is biholomorphic to the unit ball.  
This result is now known as \emph{Lu’s uniformization theorem}.  
More recently, Ebenfelt, Treuer, and Xiao~\cite{EbenfeltTreuerXiao25} extended Lu’s theorem 
by removing the completeness assumption: they showed that if the Bergman metric has constant 
negative holomorphic sectional curvature, then $D$ is biholomorphic to the unit ball 
minus a relatively closed set of measure zero, across which every $L^2$-holomorphic function 
extends holomorphically.  

\medskip

A weaker but subtler condition is the \emph{Einstein condition} for the Bergman metric.  
Recall that if $D \subset \C^n$ is a bounded 
\emph{homogeneous} domain (i.e., its biholomorphism group acts transitively on $D$), 
then the Bergman metric is automatically KE (see, e.g., \cite[p. 279]{Kobayashi1969}).  
This naturally leads to the following question, posed by Yau~\cite[Problem~43]{Yau1982ProblemSection}:

\vskip 0.3cm
\noindent
\textbf{Question~1.}  
Let $D$ be a bounded domain and $g_D$ its Bergman metric, assumed to be complete.  
Suppose that $g_D$ is Kähler–Einstein (KE).  
Is it true that $D$ must be a bounded homogeneous domain?  

\medskip

For strictly pseudoconvex domains with smooth boundary, 
Question~1 admits a positive answer, following the resolution of a conjecture 
formulated by Cheng \cite{Cheng1979}.  
Cheng conjectured that if $D \subset \C^n$ is a bounded, smoothly bounded, 
strongly pseudoconvex domain whose Bergman metric is Kähler–Einstein, 
then $D$ must be biholomorphic to the unit ball.  
This conjecture was first established in complex dimension two by 
Fu–Wong~\cite{FuWong1997} and Nemirovski–Shafikov~\cite{NemirovskiShafikov2006}, 
and subsequently proved in full generality by Huang–Xiao~\cite{HuangXiao2021}.  
It is worth mentioning that any bounded homogeneous domain with $C^2$ boundary 
is biholomorphic to the unit ball, as shown independently by 
Wong~\cite{Wong1977} and Rosay~\cite{Rosay1979}.  
More recently, the analogue of Cheng’s conjecture has been investigated 
in the broader framework of \emph{Kähler–Ricci solitons} (KRS).  
In particular, if $D\subset \C^n$ is a bounded strictly pseudoconvex 
domain with smooth boundary and $g_D$ is its Bergman metric, then Sha~\cite{Sha2024} 
proved that whenever $g_D$ is a KRS, it must in fact be biholomorphic to the unit ball.  
These developments naturally suggest the following more general problem:  

\vskip 0.3cm

\noindent\textbf{Question~2.}  
Let $D$ be a bounded domain and $g_D$ its Bergman metric.  
If $g_D$ is a KRS, must it necessarily be KE?

\vskip 0.3cm
Note that, by the Bergman–Bochner construction, the Bergman metric \(g_D\) on a bounded domain \(D\) is the pullback of the Fubini–Study metric via the Bergman map \(\Phi_D\colon D\to \mathbb{C}P^{\infty}\); hence \(g_D\) is \(\infty\)-projectively induced (see, e.g. \cite{Kobayashi1969}). 
%Here \(\Phi_D\) is constructed as follows: choose an orthonormal basis \(\{\phi_j\}_{j\ge 0}\) of the Hilbert space \(A^2(D)\) of square-integrable holomorphic functions on \(D\), and set
%\[
%\Phi_D(z)=[\phi_0(z):\phi_1(z):\cdots].
%\]
%This is well defined since \(K(z,z)=\sum_{j\ge 0}|\phi_j(z)|^2>0\) and it is independent of the chosen basis up to a unitary change; moreover,
%\[
%\Phi_D^{*}\ \omega_{\mathrm{FS}}=\frac{i}{2}\,\partial\bar\partial \log K_D(z,\bar z)=\omega_D,
%\]
%where $\omega_{\mathrm{FS}}$ is the Fubini-Study form on $\mathbb{C}P^{\infty}$ 
%and $\omega_D$
On the other hand, there exist nontrivial, indeed, complete, Kähler–Ricci solitons that are \(\infty\)-projectively induced \cite{LoiSalisZuddas22}. Therefore, in Question~2 the crucial hypothesis is the \emph{Bergman} assumption rather than mere \(\infty\)-projective inducibility of an arbitrary Kähler metric.

\vskip 0.3cm

Another remarkable feature of the Bergman metric on $\B^n$
is that if one defines
\[
K^*_{\B^n}(z, \bar z):= K_{\B^n}(z, -\bar z)= \frac{n!}{\pi^n}\,\frac{1}{(1+ |z|^2)^{\,n+1}},
\]
then 
\[
-\log K^*_{\B^n}(z, \bar z)= (n+1)\log \!\Bigl[\tfrac{n!}{\pi^n}(1+ |z|^2)\Bigr]
\]
is still a Kähler potential of a metric $g^*_{\B^n}$ globally defined on $\C^n$.  
Its associated form is
\[
\omega^*_{\B^n}=(n+1)\frac{i}{2}\partial\bar\partial\log (1+|z|^2),
\]
which is exactly $\tfrac{n+1}{\pi}$ times the Fubini–Study form $\omega_{FS}$ on $\C P^n$,
restricted to the affine chart $\C^n\cong U_0=\{[Z_0:\cdots :Z_n] \mid Z_0\neq 0\}$.  

\medskip

A similar phenomenon occurs for bounded symmetric domains.  
Recall that every bounded symmetric domain can be decomposed as a product of irreducible factors, called \emph{Cartan domains}. According to É. Cartan’s classification, Cartan domains fall into two categories: classical and exceptional.
If $\Omega \subset \C^n$ is a bounded symmetric domain with Bergman metric $g_\Omega$ and Bergman kernel 
$K_{\Omega}$, then
\[
\omega^*_{\Omega}:=-\tfrac{i}{2}\partial\bar\partial\log K^*_{\Omega},
\qquad
K^*_{\Omega }(z, \bar z)= K_{\Omega }(z, -\bar z),
\]
defines a Kähler metric $g^*_\Omega$ globally on $\C^n$.  
In fact, there exists a Hermitian symmetric space of compact type $(\Omega_c , g_c)$, the \emph{dual} of $(\Omega, g_\Omega)$, together with an integer $d$ and holomorphic embeddings
\begin{equation}\label{embeddings}
\Omega \subset\C^n\stackrel{J}{\longrightarrow} \Omega_c\stackrel{BW}{\longrightarrow} \C P^d,
\end{equation}
where $J$ has dense image and $BW$ is the Borel–Weil embedding, satisfying $BW^*(g_{FS})=g_c$.  
Moreover, $\Omega_c=J(\C^n)\sqcup H$, where $H=BW^{-1}(\{Z_0=0\})$, with $[Z_0:\dots :Z_d]$ homogeneous coordinates on $\C P^d$.  
These embeddings have been widely used in the symplectic study of bounded symmetric domains and their duals (see \cite{DISCALALOI2008sympdual}, \cite{DISCALALOIROOS2008bisympl}, \cite{LOIMOSSA2011DistExp}, \cite{LMZdual2024}).  
It turns out that $J^*g_c=\alpha g^*_\Omega$ for some $\alpha >0$.  
Thus the metric $\alpha g^*_\Omega$ is {\em finitely projectively induced}, since the holomorphic map $\varphi:=BW\circ J : \C^n\to \C P^d$ satisfies $\varphi^*g_{FS}=\alpha g^*_\Omega$.  
When $\Omega =\B^n$, we recover $\Omega_c=\C P^n$, $g_c=g_{FS}$, $BW=\mathrm{Id}$, $J$ the natural inclusion of $\C^n=U_0$ into $\C P^n$, and $\alpha=\tfrac{n+1}{\pi}$.

\medskip

This motivates the following definition.

\begin{defin}\label{dualbergman}
Let $D \subset \C^n$ be a bounded domain centered at the origin, with Bergman metric $g_D$ and kernel $K_D (z, \bar z)$ restricted to the diagonal.  
A pair $(D^*, g_D^*)$ is called the \emph{Bergman dual} of $(D,g_D)$ if:
\begin{enumerate}
\item
$D^* \subset \C^n$ is a bounded domain centered at the origin;
\item
$g_D^*$ is a Kähler metric on $D^*$ with associated form
\[
\omega^*_D =- \tfrac{i}{2}\,\partial\bar\partial\log K^*_{D}\ \  (K^*_{D}>0),
\qquad
K^*_{D}(z, \bar z):=K_{D}(z,-\bar z),
\]
on  $D \cap D^*$;
\item
$D^*$ is the maximal domain of definition of $\log K^*_{D}$.
\end{enumerate}
\end{defin}

\medskip

This leads to the following natural question:

\vskip 0.3cm
\noindent
\textbf{Question~3.} Let $D$ be a bounded domain with complete Bergman metric $g_D$.  
Assume that $(D, g_D)$ admits a Bergman dual $(D^*, g^*_{D})$ 
such that $\alpha g_D^*$ is finitely projectively induced for some $\alpha>0$.  
Must $D$ then be biholomorphic to a bounded symmetric domain?
 
\medskip

In this paper we consider the previous questions when $D$ is a   {\em Cartan–Hartogs domain} 
(CH domain in the sequel).
CH domains are a one-parameter family of noncompact  domains of $\mathbb{C}^{n+1}$, given by:
\begin{equation}\label{defCH}
M_{\Omega, \mu}:=\left\{(z, w) \in \Omega \times\mathbb{C}\, |\, \left|w\right|^2<N_{\Omega }^\mu(z, \bar{z})\right\},
\end{equation}
where $\Omega  \subset \mathbb{C}^n$ is a Cartan domain,  
known as  the {\em base} of $M_{\Omega, \mu}$,   $N_{\Omega }(z, \bar{z})$ is its generic norm,  and  $\mu>0$ is a positive real parameter.  
Recall that 
\begin{equation}\label{genericnorm}
N_\Omega(z, \overline{z})=\left(V(\Omega)K_\Omega(z, \overline{z})\right)^{-\frac{1}{\gamma}},
\end{equation}
where
$V(\Omega)$ is the {\em Euclidean volume} of $\Omega$ and $\gamma$ the {\em genus} of $\Omega$.
Notice also that $M_{\Omega, \mu}$ is homogeneous iff 
$M_{\Omega, \mu} =\B^{n+1}$ iff $\Omega=\B^n$ and $\mu=1$.

\medskip
CH domains play a central role in complex analysis and K\"ahler geometry, offering non-homogeneous yet highly structured settings for explicit computations of balanced metrics, Rawnsley's $\varepsilon$-function, Bergman kernels, and automorphisms. Key advances include \emph{explicit Bergman kernel formulas}, notably Yamamori's closed forms~\cite{YAMAMORI2012a,YAMAMORI2012b} and the construction over \emph{bounded homogeneous} bases by Ishi--Park--Yamamori~\cite{IshiParkYamamori2017}, together with related results for products and egg-type domains~\cite{AhnPark2012JFA,YINLUROOS2004newclasses}. Further progress includes the characterization of K\"ahler--Einstein metrics~\cite{FENGTU2014} and extensions to generalized CH frameworks~\cite{HAO2016}.
%Moreover, the CH version of the Polydisk Theorem established by 
%Mossa and Zedda~\cite{MOSSAZEDDApol2022} extends the classical Polydisk Theorem to this setting,
%and recent advances in the symplectic geometry of CH domains 
%by the same authors~\cite{MOSSAZEDDAsymp2022} provide further insight into their structure.  

\medskip

The main result of this paper is the following theorem, which shows that Questions~1–3 have positive answers when restricted to CH domains.

\begin{theor}\label{mainteor}
A CH domain $M_{\Omega, \mu}$ equipped with its Bergman metric  $g_{M_{\Omega, \mu}}$ admits a Bergman dual 
$(M^*_{\Omega, \mu}, g^*_{M_{\Omega, \mu}})$.  
Moreover,  the following conditions are equivalent:
\begin{itemize}
\item [(i)]
$M_{\Omega,\mu}=\B^{n+1}$;
\item [(ii)]
$g_{M_{\Omega, \mu}}$ is KE;
\item [(iii)]
$g_{M_{\Omega, \mu}}$ is a KRS with $\mu\in\Q$;
\item [(iv)]
$\alpha g^*_{M_{\Omega, \mu}}$ is finitely projectively induced,  for some $\alpha\in \R^+$.
\end{itemize}
\end{theor}

\begin{remark}\rm 
The equivalence between conditions (i) and (ii) in Theorem~\ref{mainteor}, 
that is, the fact that $M_{\Omega,\mu}$ is biholomorphic to the unit ball if and only if its Bergman metric is KE, 
has been extended in private communications (unpublished) by Yihong Hao and others 
to the broader setting of Hartogs over symmetric domains, namely those domains given by \eqref{defCH}
 whose base $\Omega$ is a bounded symmetric domain not necessarily irreducible.  
We warmly thank them for sharing this information with us.
\end{remark}

\vskip 0.3cm
The paper is organized as follows.  
In Section~\ref{proof} we prove Theorem~\ref{mainteor}.  
The argument combines explicit formulas for the Bergman kernel of Cartan–Hartogs domains with the analysis of the nonlinear ODE governing Kähler–Einstein potentials.  
A key step is to show that if the Bergman metric of a CH domain is a Kähler–Ricci soliton, then it must in fact be KE; this reduction is achieved by means of algebraic techniques inspired by Nash's theory of algebraic functions (cf. Proposition \ref{mainprop}).  
Section~\ref{maximal} investigates the maximal domain of definition of the Bergman dual of a CH domain and formulates an open problem in analogy with Calabi’s classical questions on the diastasis function.  
Finally, Section~\ref{comparison} compares the Bergman metric with other canonical metrics on CH domains, such as the metrics $g_{\Omega,\mu}$ and $\hat g_{\Omega,\mu}$ introduced in~\cite{LMZdual2024}, and discusses how our approach based on Bergman duality relates to the broader framework of Kähler duality.  
\section{Proof of Theorem \ref{mainteor}}\label{proof}

A key step in the proof of Theorem \ref{mainteor}
is the explicit expression of the Bergman kernel of $M_{\Omega, \mu}$
given in \cite{YINLUROOS2004newclasses}.
Recall that a Cartan domain $\Omega$ is uniquely determined by a triple of integers $(r,a,b)$,
where $r$ 
represents the rank of $\Omega$, namely the maximal dimension of a complex totally geodesic submanifold of $\Omega$,
and $a$ and $b$ are positive integers such that 
\begin{equation}\label{genus}
\gamma=(r-1)a+b+2
\end{equation}
and
$$n = r + \frac{r (r-1)}{2} a + r b,$$
where  $n$ is the dimension of  $\Omega$ .
Now, by 
 \cite[Corollary 3.5]{YINLUROOS2004newclasses}  
we immediately find that  the Bergman kernel for $M_{\Omega, \mu}$ is given by:

\begin{equation}\label{BerKer}
K_{M_{\Omega, \mu}}(z, w) = \frac{1}{\mu \chi(0) V} F \left( \frac{|w|^2}{N^\mu_\Omega(z, \bar z)} \right) N_\Omega^{-\gamma - \mu}(z, \bar z)
\end{equation}
where $V$ is the Euclidean volume of $\Omega$,
the function $F(X)$  is defined as (compare with   \cite[p. 14]{YINLUROOS2004newclasses} with $k = \frac{1}{\mu}$)
\begin{equation}\label{definizioneF}
F(X) =\sum_{m=0}^{\infty}  \mu(m+1)\cdot \chi\left(\mu(m+1)\right) X^m
\end{equation}
and  $\chi(s)$ is the polynomial defined as
\begin{equation}\label{defChi}
\chi(s)= \prod_{i=1}^{r} \prod_{l=0}^{b + (r-i)a} [s + 1 + (i-1) \frac{a}{2} + l].
\end{equation}

The following proposition, interesting on its own sake, allows us to pass from the KRS condition  to the KE  one
for the Bergman metric of a  CH domain.
Similar results were obtained for the case of definite or indefinite complex space forms \cite{LMpams}  and   for homogeneous \K\ manifolds \cite{PRIMO, LMrighom, LMZuniversal}.  
Recall that a \emph{\K\--Ricci soliton} (KRS) on a complex manifold $M$ is a pair 
$(g,X)$ where $g$ is a \K\ metric on $M$  and $X$ is a real holomorphic vector field on $M$
such that 
\[
\Ric(g)\;=\;\lambda\,g\;+\;\mathcal{L}_X g,
\]
for some $\lambda\in\mathbb{R}$, where $\mathcal{L}_X g$ denotes the Lie derivative of 
the metric $g$ with respect to $X$.
A KRS $(g,X)$ is \emph{trivial} if $X$ is a Killing field for $g$
(i.e.\ $\mathcal{L}_X g=0$), so that the soliton equation reduces to
\[
\Ric(g)\;=\;\lambda\,g,
\]
i.e.\ $g$ is a KE metric. 
Finally, a KRS $(g,X)$ on a complex manifold $M$ is said to be \emph{induced}
from a Kähler manifold $(N,h)$ if there exists a holomorphic map
$\varphi: M \to N$ such that $g = \varphi^{*}h$.

\begin{prop}\label{mainprop}
Let $(g,X)$ be a KRS on a complex manifold $M$, induced from 
$(M_{\Omega, \mu}, g_{M_{\Omega, \mu}})$ with $\mu \in \Q$.  
Then the soliton is \emph{trivial}, i.e. $g$ is Kähler–Einstein.
\end{prop}

In order to prove the  proposition, we need of following two lemmas.

\begin{lem}[Closed form of the Bergman kernel]\label{lem:K-closed}
Let $\chi(s)$ be the polynomial defined in \eqref{defChi}, and set
\[
P(t):=t\,\chi(t)=\sum_{m=0}^{D} a_m\,t^m \qquad\text{(so $a_0=0$ and $D=\deg P=\deg\chi+1$)}.
\]
Let $F$ be defined by \eqref{definizioneF}, namely
\[
F(X)=\sum_{j=0}^{\infty} \mu\,(j+1)\,\chi\!\big(\mu(j+1)\big)\,X^j
=\sum_{j=0}^{\infty} P\!\big(\mu(j+1)\big)\,X^j.
\]
For $(z,w)\in M_{\Omega,\mu}$ we have
$X=\frac{|w|^2}{N_{\Omega}^{\mu}(z,\bar z)}\in[0,1)$.
Then the Bergman kernel \eqref{BerKer} of $M_{\Omega,\mu}$ admits the closed form
\begin{equation}\label{eq:K-closed}
K_{M_{\Omega, \mu}}(z, w)
=\frac{N_{\Omega}^{\,\mu-\gamma}(z,\bar z)}{\mu\,\chi(0)\,V}\;
\sum_{m=0}^{D} a_m\,\mu^{m}\sum_{q=1}^{m} S(m,q)\,q!\;
\frac{|w|^{2(q-1)}}{\big(N_{\Omega}^{\mu}(z,\bar z)-|w|^{2}\big)^{q+1}},
\end{equation}
where $S(m,q)$ denotes the Stirling numbers of the second kind, with the conventions
$S(0,0)=1$ and $S(m,0)=0$ for $m\ge1$.
\end{lem}

\begin{proof}
Since $P$ has finite degree $D$, we may exchange the order of summation:
\[
F(X)=\sum_{j=0}^{\infty} P\!\big(\mu(j+1)\big)\,X^j
=\sum_{j=0}^{\infty}\sum_{m=0}^{D} a_m\big(\mu(j+1)\big)^m X^j
=\sum_{m=0}^{D} a_m\,\mu^{m}\sum_{j=0}^{\infty}(j+1)^{m} X^j .
\]
Let $(x)_{\underline{q}}=\frac{\Gamma(x+1)}{\Gamma(x-q+1)}$ be the Pochhammer symbol.
Using $x^{m}=\sum_{q=0}^{m} S(m,q)\,(x)_{\underline{q}}$,  with $x=j+1$, we get
\[
\sum_{j=0}^{\infty}(j+1)^{m} X^j
=\sum_{q=0}^{m} S(m,q)\sum_{j=0}^{\infty}(j+1)_{\underline{q}}\,X^j
=: \sum_{q=0}^{m} S(m,q)\,A_q(X).
\]
Since $(j+1)_{\underline{q}}=q!\binom{j+1}{q}$, we have
\[
A_q(X)=q!\sum_{j=0}^{\infty}\binom{j+1}{q}X^j
=
\begin{cases}
\dfrac{1}{1-X}, & q=0,\\[6pt]
q!\,X^{\,q-1}\,\dfrac{1}{(1-X)^{q+1}}, & q\ge1,
\end{cases}
\]
where for $q\ge1$ we used the index shift $j=n+q-1$ and the negative binomial series
$\sum_{n\ge0}\binom{n+q}{q}X^n=(1-X)^{-(q+1)}$.
Therefore,
\[
\sum_{j=0}^{\infty}(j+1)^{m} X^j
=\sum_{q=1}^{m} S(m,q)\,q!\,\frac{X^{\,q-1}}{(1-X)^{q+1}},
\]
since the $q=0$ term vanishes (if $m\ge1$ then $S(m,0)=0$, while for $m=0$ one has $a_0=0$).
Thus
\[
F(X)=\sum_{m=0}^{D} a_m\,\mu^{m}\sum_{q=1}^{m} S(m,q)\,q!\,
\frac{X^{\,q-1}}{(1-X)^{q+1}}.
\]
Now take $X=\dfrac{|w|^{2}}{N_{\Omega}^{\mu}(z,\bar z)}$. Since $|w|^{2}<N_{\Omega}^{\mu}(z,\bar z)$ on $M_{\Omega,\mu}$, we have $X\in[0,1)$ and
\[
1-X=\frac{N_{\Omega}^{\mu}-|w|^{2}}{N_{\Omega}^{\mu}},\qquad
X^{q-1}(1-X)^{-(q+1)}
=(N_{\Omega}^{\mu})^{2}\,
\frac{|w|^{2(q-1)}}{(N_{\Omega}^{\mu}-|w|^{2})^{q+1}}.
\]
Hence
\[
F\!\Big(\tfrac{|w|^{2}}{N_{\Omega}^{\mu}}\Big)
=(N_{\Omega}^{\mu})^{2}
\sum_{m=0}^{D} a_m\,\mu^{m}\sum_{q=1}^{m} S(m,q)\,q!\,
\frac{|w|^{2(q-1)}}{(N_{\Omega}^{\mu}-|w|^{2})^{q+1}}.
\]
By \eqref{BerKer},
\[
K_{M_{\Omega,\mu}}(z,w)
=\frac{1}{\mu\,\chi(0)\,V}\;F\!\Big(\tfrac{|w|^{2}}{N_{\Omega}^{\mu}}\Big)\;
N_{\Omega}^{-\gamma-\mu},
\]
and since $(N_{\Omega}^{\mu})^{2}\,N_{\Omega}^{-\gamma-\mu}=N_{\Omega}^{\mu-\gamma}$,
we obtain \eqref{eq:K-closed}.
\end{proof}

Let us recall that if $g$ is a Kähler metric on a complex manifold $M$, then the diastasis function $D^g_q$ of $g$ centered at a point $q \in M$ (see \cite{CALABI1953diast}) is the unique Kähler potential with the property that, for any holomorphic coordinate system on $M$ centered at $q$, its power series expansion at the origin contains no purely holomorphic or purely antiholomorphic terms.  

Given a coordinate system centered at $q$ and a Kähler potential $\phi: U \to \R$ for $g$ defined on a neighborhood $U$ of $q$, let $\tilde{\phi}$ denote its analytic continuation to a neighborhood of the diagonal in $U \times U$. Then (see \cite{CALABI1953diast})
\begin{equation}\label{eq:diaCal}
D^g_q(z)=\tilde \phi(z,\bar z)+\tilde \phi(0,0)-\tilde \phi(z,0) - \tilde \phi(0,\bar z).
\end{equation}

Before proceeding, let us recall the following notion (see \cite{PRIMO}, \cite{LMrighom}):  
a holomorphic function $f$ on a complex manifold $M$ is said to be a \emph{holomorphic Nash algebraic function} (in short \emph{Nash}) if, for every point $p \in M$, there exists a neighborhood $U$ of $p$ and a nonzero polynomial $P(z,y)$ in the variables $z \in U$ and $y \in \C$ such that
$$
P(z, f(z)) \equiv 0 \quad \text{for all } z \in U.
$$
In other words, $f$ is holomorphic and algebraic over the ring of holomorphic functions on $M$.

\begin{lem}\label{lem:diast}
Let $p \in M_{\Omega, \mu}$, and let $\D_p$ denote the diastasis function of $g_{M_{\Omega, \mu}}$ centered at $p$.  
If $\mu \in \Q$, then $e^{\D_p}$ is Nash in the variables $z, \bar z, w, \bar w$.  
%In particular, if $p=0=(0,0)\in M_{\Omega, \mu}$ is the origin, one has
%\begin{equation}\label{eq:diaor}
%\D_0(z,w)=\log K_{M_{\Omega, \mu}}(z, w).
%\end{equation}
\end{lem}

\begin{proof}
It is well known (see, e.g., \cite[Sec.~1.3]{YINLUROOS2004newclasses}) that the generic norm of a Cartan domain $\Omega$ can be written as
$$
N_{\Omega}(z, \bar{z})=1-m_1(z, \bar{z})+m_2(z, \bar{z})-\cdots+(-1)^r m_r(z, \bar{z}),
$$
where $m_1, \ldots, m_r$ are polynomials on $\Omega \times \bar{\Omega}$, homogeneous of respective bidegrees $(1,1), \ldots,(r,r)$.  
This fact, together with \eqref{eq:K-closed} and \eqref{eq:diaCal}, shows that for $\mu \in \Q$, $K_{M_{\Omega, \mu}}(z, w)$ (which is the \K\ potential for the Bergman metric $g_{M_{\Omega, \mu}}$)   is Nash in $z, \bar z, w, \bar w$.
%i.e., it proves \eqref{eq:diaor}
Thus, if $\mu \in \Q$, then $\D_p$ is Nash in $z, \bar z, w, \bar w$. The proof is complete.
\end{proof}

We are now in a position to prove Proposition \ref{mainprop}.

\begin{proof}[Proof of Proposition \ref{mainprop}]
Let $(M,g)$ be a Kähler submanifold of $(M_{\Omega, \mu}, g_{M_{\Omega, \mu}})$.  
By the hereditary property of the diastasis function (see \cite{CALABI1953diast}), the diastasis function of $g$ centered at a point $p \in M$ is given by
\[
D_p^g = \D_{f(p)} \circ f,
\]
where $f: M \to M_{\Omega, \mu}$ is a holomorphic map such that $g = f^* g_{M_{\Omega, \mu}}$.  
Combining this with Lemma \ref{lem:diast}, and using the notation introduced in \cite{PRIMO}, we obtain that, in any coordinate system on $M$ centered at $p$, one has
$
D_p^g \in \widetilde{\mathcal{F}}.
$
Hence, by \cite[Proposition~4.1]{PRIMO}, it follows that any KRS $(g,X)$ on $M$ is \emph{trivial}, i.e., $g$ is Kähler--Einstein.
\end{proof}

\begin{remark}\rm
We believe that Proposition \ref{mainprop} should remain valid for every $\mu > 0$.  
However, when $\mu$ is irrational, the analytic continuation of $e^{D_p^{M_{\Omega, \mu}}}$ cannot be expressed as a product of powers of Nash functions.  
This prevents the direct application of \cite[Proposition~4.1]{PRIMO}, and therefore the validity of Proposition \ref{mainprop} cannot, at present, be extended to all $\mu > 0$.
\end{remark}

\begin{proof}[Proof  of Theorem \ref{mainteor}]
To prove 
$(i)\iff (ii)$ it is enough to show that
$(ii)\Rightarrow (i)$
since the Bergman metric on the unit ball is KE.

Now,  
\cite[Lemma 5]{WANGYINZHANGROSS2006KEonHartogs} asserts that 
a \K\  metric $g$ with associated \K\  form   $\omega=\frac{i}{2\pi}\partial\bar\partial \Phi$ on the CH domain
$M_{\Omega, \mu}$ is KE, with Einstein constant $-(n+2)$,  if and only if 

\begin{equation}\label{potentialtype}
\Phi(z, w)=h(X) -\frac{\gamma + \mu}{n+ 2} \log N_\Omega(z, \bar z),
\end{equation}
where $X = \frac{|w|^2}{N^\mu_\Omega(z, \bar z)}$ and $h$ satisfies the differential equation

\begin{equation}\label{ODE}
\left( \mu X h'(X) + \frac{\gamma + \mu}{n + 2} \right)^n [X h'(X)]' = \delta e^{(n+2)h(X)},
\end{equation}
for some $\delta \in \R$.
By \eqref{BerKer} a \K\  potential  for the Bergman metric $g_{M_{\Omega, \mu}}$ is then given by 
$$ -(\gamma+\mu )\log N_\Omega(z, \bar z) + \log F \left(X \right) $$

and it becomes of the form (\ref{potentialtype}) if we take the multiple $\frac{1}{n+2} g_{M_{\Omega, \mu}}$, for
\begin{equation}\label{definizionehF}
h(X) =\frac{1}{n+2} \log F \left( X \right).
\end{equation}

Then the proof will be achieved if we prove that if  $h(X)$ satisfies the ODE \eqref{ODE}
then $\Omega$ is the unit ball and $\mu=1$.
In order to do that, let us denote by $D$ (as in Lemma \ref{lem:K-closed})  the degree of $(m+1) \mu \cdot \chi\left(\mu(m+1)\right)$ as a polynomial in $m$, being $\chi$ the polynomial function in (\ref{definizioneF}).
Then  there exist  real coefficients $b_0, b_1, \dots, b_D$ such that
\begin{equation}\label{decompolyn}
(m+1) \mu \cdot \chi\left(\mu(m+1) \right) = \sum_{k=0}^D b_k \frac{m!}{(m-k)!}
\end{equation}
Then, by (\ref{definizioneF}) we can write

$$F(X) = \sum_{k=0}^D  b_k \sum_{m=0}^{\infty} \frac{m!}{(m-k)!} X^m.$$

By using  $\frac{k! X^k}{(1 - X)^{k+1}} = \sum_{m=0}^{\infty} \frac{m!}{(m-k)!} X^m$ (which can be obtained by derivations of $(1 - X)^{-1} = \sum_{m=0}^{\infty} X^m$) we finally get

\begin{equation}\label{Frational}
F(X) = \sum_{k=0}^D b_k \frac{k! X^k}{(1 - X)^{k+1}} = \frac{Q(X)}{(1 - X)^{D+1}}
\end{equation}
where we set
\begin{equation}\label{definQ}
Q(X) := \sum_{k=0}^D b_k k! X^k (1 - X)^{D - k}
\end{equation}
This gives  an explicit representation of $F$ as rational function (with  the $b_k$'s given by (\ref{decompolyn})).
Notice that by (\ref{decompolyn}) and since we are assuming that the degree of $(m+1) \mu \cdot \chi((m+1) \mu)$ is exactly $D$ we have $Q(1) = b_D D! \neq 0$.
Analogously, we have also $Q(0) = b_0 \neq 0$: indeed, by (\ref{decompolyn}) this is equivalent to say that $(m+1) \mu \cdot \chi((m+1) \mu)$, as polynomial in $m$, has nonvanishing constant term. But this is true since, by the very definition  of $\chi$ we have that this constant term is $\mu \prod_{i=1}^{r} \prod_{l=0}^{b + (r-i)a} [\mu +1+ (i-1) \frac{a}{2} + l]$.

It follows by  \eqref{definizionehF} and \eqref{Frational} that 

$$h(X) =\frac{1}{n+2} \log F(X)= \frac{1}{n+2} \log \frac{Q(X)}{(1-X)^{D+1}}.$$

By inserting this expression in  the ODE (\ref{ODE}) one gets that 
$$h'(X) = \frac{1}{n+2} \frac{Q'(X)}{Q(X)} + \frac{D+1}{n+2} \frac{1}{1-X}$$
and
$$[X h'(X)]' = \frac{1}{n+2} \frac{(Q'(X) + XQ''(X)) Q(X) - X Q'(X)^2}{Q(X)^2} + \frac{D+1}{n+2} \frac{1}{(1-X)^2}$$

Then, after a straightforward computation   (\ref{ODE}) yields

\begin{equation}\label{ODEbergman}
\frac{U^n(X) V(X)}{W(X)} = \frac{\delta Q(X)}{(1-X)^{D+1}}, 
\end{equation}
where 
\begin{equation}\label{U(X)}
U(X)=\mu X (1-X) Q'(X)+ \mu X (D+1) Q(X) + (\gamma + \mu) (1-X) Q(X), 
\end{equation}
\begin{equation}\label{V(X)}
V(X)=\left[(Q'(X)+ X Q''(X)) Q(X) - X Q'^2(X)\right](1-X)^2 + (D+1) Q^2(X), 
\end{equation}
and 
\begin{equation}\label{W(X)}
W(X)=(n+2)^{n+1}Q(X)^{n+2}(1-X)^{n+2}. 
\end{equation}

Notice that the numerator of the left hand-side of  \eqref{ODEbergman}, i.e. $U(X)^n V(X)$, tends to $\mu^n(D+1)^{n+1} Q(1)^{n+2} \neq 0$ for $x \rightarrow 1$ (indeed, by the definition (\ref{definQ}) of $Q(X)$ we immediately see that $Q(1) = b_D D!$ which is not zero as observed above.
Then, by comparing the behaviour for $X \rightarrow 1$ of $W(X)$ and the denominator of  the right-hand side of \eqref{ODEbergman}, i.e. $(1-X)^{D+1}$,   we deduce that  $n+2 = D+1$.
By replacing in (\ref{ODEbergman}) and simplifying we get
\begin{equation}\label{ODEbergman3}
U^n(X) V(X) =  \delta (n+2)^{n+1}Q(X)^{n+3}. 
\end{equation}

Now, we claim that $Q(X)$ is a  constant polynomial.
Assume by contradiction that this is not true  and let $X_0$, $X_0\neq 0, 1$,  be a (possibly complex) root of $Q$, with multiplicity $k \geq 1$. Then $Q(X) = (X - X_0)^k \tilde Q(X)$ for some polynomial $\tilde Q(X)$ such that $\tilde Q(X_0) \neq 0$.
Then, by replacing 
$$Q'(X) = k (X- X_0)^{k-1} \tilde Q(X) + (X - X_0)^k \tilde Q'(X)$$
$$Q''(X) = k(k-1) (X- X_0)^{k-2} \tilde Q(X) + 2 k (X- X_0)^{k-1} \tilde Q'(X) + (X - X_0)^k \tilde Q''(X)$$
into \eqref{U(X)} and \eqref{V(X)}
one gets respectively 
$$U^n(X)= (X-X_0)^{(k-1) n} A(X)$$
and 
$$V(X)= (X-X_0)^{2k-2} B(X),$$
where $A(X_0) = \mu^n X_0^n (1-X_0)^n k^n \tilde Q^n(X_0) \neq 0$ 
and $B(X_0) = -X_0 k \tilde Q^2(X_0) \neq 0$.

So we see that the left-hand side of (\ref{ODEbergman3}), i.e. $U^n(X) V(X)$,  vanishes for $X \rightarrow X_0$ with order 
$$(k-1)n + (2k - 2) = (n+2)(k-1)$$
while the right-hand side, i.e. $\delta (n+2)^{n+1}Q^{n+3}(X)$,  vanishes with order $(n+3)k$.

But the equality $(n+2)(k-1) = (n+3) k$ is impossible, so the claim that $Q(X)$ is a constant follows.
Then, by (\ref{Frational}), it follows that
\begin{equation}\label{Fequal}
F(X) = \frac{c}{(1-X)^{D+1}}= \frac{c}{(1-X)^{n+2}}
\end{equation}
for some constant $c$.
Then the conclusion, and hence the implication  is obtained by the following lemma,
with $d=n+2$.

\begin{lem}\label{mainlemma}
Assume that there exist a real constant $c$ and a positive integer $d$ such that 
the function $F(X)$ given by \eqref{definizioneF} is of the form
\begin{equation}\label{Fcd}
F(X) = \frac{c}{(1-X)^{d}}.
\end{equation}
Then $M_{\Omega, \mu}=\C H^{n+1}$ and $d=n+2$.
\end{lem}
\begin{proof}
 By the definition of $F$

\begin{equation}\label{defFbis}
F(X) = \mu \chi(\mu) + 2 \mu \chi(2 \mu) X + 3 \mu \chi(3\mu) X^2 + \cdots
\end{equation}

and the power series expansion of $(1 - X)^{- d}$

\begin{equation}\label{defF22power}
(1- X)^{-d} = 1+  d X +  \frac{d(d+1)}{2!} X^2 +  \frac{d(d+1)(d+2)}{3!} X^3 + \cdots
\end{equation}

we see that (\ref{Fcd}) implies the following condition, which must be satisfied for any positive integer $s \geq 3$:

\begin{equation}\label{condChi}
s \mu \cdot \chi(s \mu) = c \frac{d(d+1) \cdots (d + s - 2)}{(s- 1)!}
\end{equation}

i.e., by definition of $\chi$,

\begin{equation}\label{condChi2}
\prod_{i=1}^{r} \prod_{l=0}^{b + (r-i)a} [s \mu + 1 + (i-1) \frac{a}{2} + l] = c \frac{d(d+1) \cdots (d + s - 2)}{s! \mu}
\end{equation}

In order to compare more easily the sides of this equality, let us rewrite the right-hand side as 

\begin{equation}\label{condChi3b}
\prod_{i=1}^{r} \prod_{l=0}^{b + (r-i)a} [s \mu + 1 + (i-1) \frac{a}{2} + l] = c \frac{(s + d - 2)(s + d - 3) \cdots (s + 1)}{(d -1)! \mu}
\end{equation}

As we have observed above, this equality is true for any integer $s \geq 3$, anyway since both sides are polynomials in $s$, this must be true for any $s \in \R$.

Now we prove that if   $r > 1$ then (\ref{condChi3b}) does not hold true.
Indeed  we claim that under this assumption, the left-hand side of (\ref{condChi3b}) has either multiple roots or non integral roots: since the right-hand side clearly has always only integral distinct roots, this will prove our assertion.
Let us first notice that the roots of the polynomial on the left-hand side of (\ref{condChi3b}) are $s = s(i, l) := - \frac{1}{\mu} - [(i-1) \frac{a}{2} + l] \frac{1}{\mu}$, and, by $a, \mu \geq 0$, $i \geq 1$ and $l \geq 0$, one gets $s(i, l) \leq s(1, 0) = - \frac{1}{\mu}$. So $s = - \frac{1}{\mu}$ is the greatest root of the polynomial, but by comparing with the right-hand side of (\ref{condChi3b}) one finds that it must be $- \frac{1}{\mu} = -1$, i.e. $\mu = 1$.
In order to prove the claim, set $B:= b + (r-1)a + 1$: since $r \geq 2$, in the left-hand side of (\ref{condChi3b}) we have factors with $i=1$ and $i=2$. More precisely, for $i=1$, and taking into account that $\mu = 1$, we get

\begin{equation}\label{prodi=1CORRECT}
\prod_{l=0}^{b + (r-i)a} [s  + 1 + (i-1) \frac{a}{2} + l]  = \prod_{l=0}^{b + (r-1)a} [s + 1 + l]  =  (s + 1)  \cdots (s + B)
\end{equation}

which yields the roots $s = -1, -2, \dots, -B$.
On the other hand, for $i=2$ we have

\begin{equation}\label{prodi=2CORRECT}
\prod_{l=0}^{b + (r-i)a} [s  + 1 + (i-1) \frac{a}{2} + l]  = \prod_{l=0}^{b + (r-2)a} [s + 1 + \frac{a}{2} + l]  =  (s + 1 +  \frac{a}{2}) \cdots (s + B -  \frac{a}{2})
\end{equation}

In particular, $s = - 1 - \frac{a}{2}$ is a root, and we have two possibilities: either $\frac{a}{2} \notin \Z$, and we have finished, or $\frac{a}{2} \in \Z$.
In the latter case, the root $-1 - \frac{a}{2}$ is one of the roots $-1, -2, \dots, -B$ given by (\ref{prodi=1CORRECT}) if and only if $-1 - \frac{a}{2} \geq - B$. But, by definition of $B$, this means
$- 1 - \frac{a}{2} \geq -b -(r-1)a - 1$
i.e.
$b + (r - \frac{3}{2})a \geq 0$
which holds true under the assumption $r \geq 2$ since $a, b \geq 0$. Then the left-hand side of (\ref{condChi3b}) has multiple roots. The claim is proved.
Now, the only bounded symmetric space with $r=1$ is the complex hyperbolic space $\C H^n$, for which one has also $a=0$ and $b= n-1$: then (\ref{condChi3b}) reads
\begin{equation}\label{condChi3}
\prod_{l=0}^{n-1} [s + 1 + l] = c \frac{(s + d - 2)(s + d - 3) \cdots (s + 1)}{(d -1)! }
\end{equation}
which is verified only for  $d = n+2$ and $c = (n+1)!$.
\end{proof}

\vskip 0.1cm
To prove 
$(ii)\iff (iii)$
 it is enough to show that
$(iii)\Rightarrow (ii)$
which follows by Proposition \ref{mainprop}.

\vskip 0.1cm

To prove 
$(i)\iff (iv)$ it is enough to show that
$(iv)\Rightarrow (i)$
since the dual  Bergman metric  of the unit ball is
the restriction to $U_0\cong\C^n$ of the Fubini-Study metric.
Notice that the \K\ form $\omega_{M_{\Omega, \mu}}$ associated to the  Bergman metric $g_{M_{\Omega, \mu}}$ 
is given by
$$\omega_{\Omega, \mu} = \frac{i}{2}\partial \bar \partial \log K_{M_{\Omega, \mu}}(z, \bar z)= \frac{i}{2} \partial \bar \partial \log \left[ F \left( \frac{|w|^2}{N_{\Omega}^{\mu}(z, \bar z)} \right) N_{\Omega}^{-\gamma - \mu} (z, \bar z)\right].$$
Since $N_{\Omega}^{\mu}(z, -\bar z)$ is real-valued, it follows that 
$$\omega_{M_{\Omega, \mu}}^* = -\frac{i}{2} \partial \bar \partial \log \left[ F \left( -\frac{|w|^2}{N_{\Omega}^{\mu}(z, -\bar z)} \right) N_{\Omega}^{-\gamma - \mu} (z, -\bar z)\right]$$
is a \K\ form on a suitable (maximal) neighborhood  $M_{\Omega, \mu}^*$ of the origin of $\C^{n+1}$.
Thus the metric $g_{M_{\Omega, \mu}}^*$ associated to $\omega_{M_{\Omega, \mu}}^*$ is then the metric dual to $g_{M_{\Omega, \mu}}$.
We have to show that if $\alpha g_{M_{\Omega, \mu}}^*$ is finitely projectively induced for some $\alpha\in\R^+$ then $M_{\Omega, \mu}=\C H^{n+1}$.
Now, if $\alpha g_{M_{\Omega, \mu}}^*$ is finitely projectively induced then   the same is clearly true also for the submanifold defined by $z=0$, which has restricted metric
whose associated \K\ form is given by
$$-\alpha \frac{i}{2} \partial \bar \partial \log   F \left( -|w|^2 \right), $$
where we have used the fact that $N_\Omega(0, 0)= 1$.
Since the potential $D(w) = - \alpha \log F(- |w|^2)$ for this metric  is the Calabi's diastasis around the origin, then by the Calabi's criterium \cite{CALABI1953diast} if the metric is finitely projectively induced then 
$$e^{D(w)} = \frac{1}{F^\alpha(-|w|^2)}$$
must have a finite expansion in $|w|^2$.
By setting $X:=|w|^2$ and using \eqref{Frational} and \eqref{definQ}  this means that there exists a polynomial 
$P(X)$ such that 
\begin{equation*}
P(X)=\frac{1}{F^\alpha(-X)} =  \frac{(1 + X)^{(D+1)\alpha}}{Q^{\alpha}(-X)}, 
\end{equation*}
namely, 
\begin{equation}\label{EQ1+x}
(1 + X)^{(D+1)\alpha} = P(X)Q^{\alpha}(-X).
\end{equation}

We are going to show that from the previous equality  one can deduce that $Q(-X)$ and hence $Q(X)$ is a constant $c$ and hence,  by \eqref{Frational}
$F(X)=\frac{c}{(1-X)^d}$ with $d=D+1$ and Lemma \ref{mainlemma},  we get that $M_{\Omega, \mu}=\C H^{n+1}$ and 
\begin{equation}\label{dD+1n+1}
d=D+1=n+2,
\end{equation}
as desired.
In order to show that $Q(X)$ is constant, notice that from  \eqref{EQ1+x}  and  $Q(1) = b_D D!\neq 0$ we deduce that $P(-1)=0$.
Now, if we derivate (\ref{EQ1+x}) we get

\begin{equation}\label{EQ1+x2}
(D+1)\alpha (1 + X)^{(D+1)\alpha-1} = P'(X) Q^{\alpha}(-X) + P(X) [Q^{\alpha}(-X)]'.
\end{equation}

Then if  $(D+1)\alpha > 1$, by using $P(-1) = 0$ and $Q(1) \neq 0$ we now deduce $P'(-1) = 0$.

We can then use the formula for the the $k$-th derivative

\begin{equation}\label{EQ1+xk}
\begin{aligned}
&(D+1)\alpha \bigl[(D+1)\alpha - 1 \bigr] \cdots \bigl[(D+1)\alpha - k + 1 \bigr]
  (1 + X)^{(D+1)\alpha - k} \\
&\quad = \sum_{j=0}^k \binom{k}{j} \, P^{(j)}(X) \, \bigl(Q^{\alpha}(-X)\bigr)^{(k-j)}.
\end{aligned}
\end{equation}
in order to deduce by induction that for every positive integer $k$ such that $(D+1) \alpha - k > 0$ one has $P^{(k)}(-1) = 0$.
Now, if $(D+1) \alpha \notin \Z$, let $k_0$ be the first positive integer such that $(D+1) \alpha < k_0$: then, from (\ref{EQ1+xk}) for $k = k_0$ we see that the left-hand side tends to $\infty$ for $X \rightarrow -1$, while the right-hand side clearly does not (recall that $Q(1) \neq 0$).
It follows that it must be 
\begin{equation}\label{Dalphako}
(D+1) \alpha = k_0 \in \Z^+, 
\end{equation}
and by the above we can say that $P^{(k)}(-1) = 0$ for $k \leq k_0 - 1$, so that $P(X) = (1+X)^{k_0} \tilde P(X)$,  for some polynomial $\tilde P$.
Then, by  (\ref{EQ1+x}) we deduce that 
$\tilde P(X) = Q^{-\alpha}(-X)$ and since  $\alpha > 0$ and $Q(-X)$ and $\tilde P(X)$ are polynomials
this forces $Q(X)$ to be constant and we are done.
%Notice that \eqref{dD+1n+1} and \eqref{Dalphako} gives $(n+2)\alpha$ with $\alpha\in\Z^+$
%for the Fubini-Study metric on  $\C P^{n+1}$ to be projectively induced, in accordance with the  example
%of the ball in the introduction.
\end{proof}

\section{On the maximal domain of  definition of the Bergman dual of a CH domain}\label{maximal}
The problem of determining the maximal domain of definition of $-\log K^*_{M_{\Omega,\mu}}$ appears to be a deep and difficult issue, reminiscent of Calabi’s classical questions about the diastasis function (see \cite{CALABI1953diast}).  
In particular we have the  following open problem for CH domains (which can also be formulated more generally for bounded domains, with the ball replaced by a bounded symmetric domain):  

\vskip 0.3cm

\noindent
\textbf{Question~4.}  
Let $(M^*_{\Omega,\mu}, g^*_{M_{\Omega,\mu}})$ be the Bergman dual of a CH domain.  
Assume that $M^*_{\Omega,\mu}=\C^{n+1}$.  
What can be said about $M_{\Omega,\mu}$?  
Moreover, under the same assumption, suppose that $(\C^{n+1}, g^*_{M_{\Omega,\mu}})$ admits a compactification, i.e., there exists a compact Kähler manifold $(N,g)$ such that $\C^{n+1}$ is a dense subset of $N$ and $g|_{\C^{n+1}}=g^*_{M_{\Omega,\mu}}$.  
Is it then true that $M_{\Omega,\mu}=\B^{n+1}$?  

\vskip 0.3cm

Here we content ourself with an explicit example. 

\begin{example}\rm
We exhibit a complex   $2$-dimensional CH domain whose Bergman dual is not defined on all of $\C^2$ 
and whose dual metric admits no compactification.
Let the base be the rank--one Cartan domain $\Omega=\B^1\subset\C$, so that
\[
N_\Omega(z,\bar z)=1-|z|^2, \qquad \gamma=2,
\]
and fix a parameter $\mu>1$.  
For $\Omega=\B^1$,  $\chi(s)=s+1$,
\[
P(t):=t\,\chi(t)=t(t+1)=t^2+t,
\]
and the generating function in \eqref{definizioneF} has the closed rational form
\[
F(X)
=\sum_{j\ge0}\mu(j+1)\chi(\mu(j+1))\,X^j
=\frac{\mu\bigl((1+\mu)+(\mu-1)X\bigr)}{(1-X)^3}.
\]
Therefore, by \eqref{BerKer},
\begin{equation}\label{eq:Example-BergmanKernel}
K_{M_{\Omega,\mu}}(z,w)
=\frac{C_\mu}{\bigl(1-|z|^2\bigr)^{\mu+2}}\,
\frac{(1+\mu)+(\mu-1)X}{(1-X)^3},\qquad
X=\frac{|w|^2}{\bigl(1-|z|^2\bigr)^{\mu}},
\end{equation}
for a positive constant $C_\mu=\bigl(\mu\,\chi(0)\,V\bigr)^{-1}$ (here $\chi(0)=1$ and $V=\Vol(\B^1)$).

By Definition~\ref{dualbergman}, the Bergman dual is obtained by replacing $\bar z$ with $-\bar z$ 
in \eqref{eq:Example-BergmanKernel}. Hence
\begin{equation}\label{eq:Kstar-disk}
K^*_{M_{\Omega,\mu}}(z,w)
=\frac{C_\mu}{\bigl(1+|z|^2\bigr)^{\mu+2}}\,
\frac{(1+\mu)-(\mu-1)Y}{(1+Y)^3},
\qquad
Y=\frac{|w|^2}{\bigl(1+|z|^2\bigr)^{\mu}}.
\end{equation}
Observe that $K^*_{M_{\Omega,\mu}}>0$ if and only if
\[
(1+\mu)-(\mu-1)Y>0
\;\;\Longleftrightarrow\;\;
Y<\frac{1+\mu}{\mu-1}.
\]
Consequently, the maximal domain of definition $U_{\max}^*$ of $-\log K^*_{M_{\Omega,\mu}}$ satisfies
\begin{equation}\label{eq:Umax-star}
U_{\max}^*\subset\Bigl\{(z,w)\in\C^2:\; |w|^2
<\frac{1+\mu}{\mu-1}\,\bigl(1+|z|^2\bigr)^{\mu}\Bigr\}.
\end{equation}
This is a proper domain in $\C^2$ (for instance, for $z=0$ one needs $|w|^2<\tfrac{1+\mu}{\mu-1}$), hence
$U_{\max}^*\subsetneq\C^2$.

Let $z=x+iy$, $w=u+iv$, and consider the real $2$--plane
\[
\Pi=\mathrm{span}\{\partial_x,\partial_v\}.
\]
A direct computation of the Levi--Civita connection and the Riemann curvature tensor 
gives
\[
K_{\Pi}(0,iv)=
\frac{v^{8}+12v^{6}-26v^{4}+12v^{2}-15}
     {12\,(v^{2}-1)^{2}\,(v^{4}-2v^{2}+5)}.
\]
In particular,
\[
\lim_{v\to 1^\pm}K_{\Pi}(0,iv)=-\infty,
\]
so the sectional curvature of $(U_{\max}^*,\omega^*_{M_{\C H^1,2}})$ is unbounded from below 
when approaching the boundary $|v|=1$ with $z=0$.
By contrast, the sectional curvature of a compact Riemannian manifold is necessarily 
bounded, and by the Gauss equation any totally geodesic submanifold inherits this property. 
Hence $\bigl(U_{\max}^*,\omega^*_{M_{\C H^1,2}}\bigr)$
admits no totally geodesic immersion into any compact Riemannian manifold.

\end{example}

\section{Comparison with other canonical metrics on CH domains}\label{comparison}

Our Theorem \ref{mainteor} should be compared with the recent work \cite{LMZdual2024}, where the notion of \emph{Kähler duality} is introduced and studied using Calabi’s diastasis function instead of the Bergman kernel
(the reader is referred to \cite{LMZdual2024} for the definition of \K\ dual).  
In that setting, two other natural complete Kähler metrics on $M_{\Omega,\mu}$ are considered, denoted by $g_{\Omega,\mu}$ and $\hat g_{\Omega,\mu}$.  
The first metric, originally introduced by Roos, Wang, Yin, L. Zhang, and W. Zhang \cite{WANGYINZHANGROSS2006KEonHartogs}, has associated Kähler form  
\begin{equation}\label{eqKobform}
\omega_{\Omega,\mu} \;=\; -\frac{i}{2\pi}\,\partial\bar\partial \log \bigl(N_\Omega^\mu(z,\bar z) - |w|^2\bigr).
\end{equation}
This metric is Kähler–Einstein precisely when $\mu = \tfrac{\gamma}{n+1}$, where $n$ is the complex dimension of $\Omega$ and $\gamma$ its genus (see \cite[Th. 1.1]{LMZdual2024}).  
The second metric $\hat g_{\Omega,\mu}$ on a CH domain $M_{\Omega,\mu}$ introduced in \cite{LMZdual2024} is defined by the Kähler form
\begin{equation}\label{modmetric}
\hat\omega_{\Omega,\mu} \;=\; -\frac{i}{2\pi}\,\partial\bar\partial \log \bigl(N_\Omega^\mu(z,\bar z) - |w|^2\bigr) 
-\frac{i}{2\pi}\,\partial\bar\partial \log N_\Omega^\mu(z,\bar z).
\end{equation}
The pair $(M_{\Omega,\mu}, \hat g_{\Omega,\mu})$ admits a Kähler dual 
$(M_{\Omega,\mu}^*=\C^{n+1}, \hat g_{\Omega,\mu}^*)$, and for $\mu \in \Z^+$ and sufficiently large integers $\alpha$, the metric $\alpha \hat g_{\Omega,\mu}$ is projectively induced (see \cite[Th. 1.2]{LMZdual2024}). These results highlight the importance of focusing specifically on the Bergman metric, rather than on arbitrary Kähler metrics, and of using the Bergman duality  in order to obtain the conclusion of Theorem~\ref{mainteor}.  
Notice also that when considering the \K\ metric $\hat g_{\Omega,\mu}$, its Kähler dual $(M_{\Omega,\mu}^*=\C^{n+1}, \hat g_{\Omega,\mu}^*)$ does admit a compactification (see \cite[Th. 1.2]{LMZdual2024}).  
This further emphasizes  (cf. Question 4)  the distinctive role of the Bergman metric and of Bergman duality.
The following proposition shows a comparison between   the Bergman metric on CH with these two metrics.

\begin{prop}
Let $M_{\Omega, \mu}$ be a CH domain.
Then 
\begin{enumerate}
\item
if $g_{M_{\Omega, \mu}}=\alpha g_{\Omega, \mu}$, for some $\alpha\in \R^+$ then $M_{\Omega, \mu}=\B ^{n+1}$ and $\alpha=n+2$;
\item
it cannot exist $\alpha\in \R^+$
such that 
$g_{M_{\Omega, \mu}}=\alpha \hat g_{\Omega, \mu}$.
\end{enumerate}
\end{prop}
\begin{proof}
As observed in  \cite{YINLUROOS2004newclasses}, the function $F$ given by \eqref{definizioneF} is a finite linear combination of derivatives of $\frac{1}{1-X}$, so we can also write

\begin{equation}\label{defF2}
F(X) =  \sum_{i=1}^{m} c_i (1- X)^{-i}
\end{equation}
for some real constants $c_1, \dots, c_m$.
 Thus, by (\ref{defF2}), the Bergman kernel of $M_{\Omega, \mu}$ can be also written as 

\begin{equation}\label{BerKer2}
K_{M_{\Omega, \mu}}(z, w) = \frac{1}{\mu \chi(0) V} \sum_{j=1}^m c_j \left( \frac{N^\mu_\Omega(z, \bar z) - |w|^2}{N^\mu_\Omega(z, \bar z)} \right)^{-j} N^{-\gamma-\mu}_\Omega(z, \bar z)
\end{equation}

By $$\omega_{\Omega, \mu}=  -\frac{i}{2} \partial \bar \partial \log [N(z, \bar z)^{\mu} - |w|^2 ]$$
the assumption $g_{M_{\Omega, \mu}}=\alpha g_{\Omega, \mu}$
is equivalent to 

$$\frac{i}{2} \partial \bar \partial \log K_{M_{\Omega, \mu}}(z, w)= -\alpha\frac{i}{2} \partial \bar \partial \log [N^\mu_\Omega(z, \bar z) - |w|^2 ].$$

By \eqref{BerKer2} and  standard argument  one gets

\begin{equation}\label{keyequality}
\sum_{j=1}^m c_j \left( \frac{N^\mu_\Omega(z, \bar z) - |w|^2}{N^\mu_\Omega(z, \bar z)} \right)^{-j} N^{-\gamma-\mu} _\Omega(z, \bar z)= c\ [N^\mu_\Omega(z, \bar z) - |w|^2]^{-\alpha},
\end{equation}
for some constant $c$.
Then, (\ref{keyequality}) can be rewritten as 
\begin{equation}\label{keyequality2}
\sum_{j=1}^m c_j  \frac{N_\Omega^{\mu j}(z, \bar z)}{(N_\Omega^{\mu}(z, \bar z) - |w|^2)^{j-\alpha}} = c \ N_\Omega^{\gamma+\mu}(z, \bar z).
\end{equation}

Since the right-hand side of this equality does not depend on $|w|^2$, its derivative w.r.t. $|w|^2$ is identically  zero, i.e.
$$\sum_{j=1}^m   \frac{c_j (j-\alpha)N_\Omega^{\mu j}(z, \bar z)}{(N^{\mu}_\Omega(z, \bar z) - |w|^2)^{j-\alpha+1}} \equiv 0$$

By setting $|w|^2 = 0$, we get

\begin{equation}\label{sum}
\sum_{j=1}^m c_j  (j-\alpha) N^{\mu (\alpha-1)}_\Omega(z, \bar z) \equiv 0
\end{equation}

from which we deduce that the sum must have only one addendum $c_j$ with $j = \alpha \in \Z$. Then, by \eqref{keyequality2} one deduces that 
$c_{\alpha} N^{\mu(\alpha -1) - \gamma}_\Omega(z, \bar z) =c $
from which we conclude $c_\alpha=c$ and
\begin{equation}\label{relationgmul}
\alpha = \frac{\gamma + \mu}{\mu}
\end{equation}
Now, the condition that the sum in (\ref{sum}) must have only one addendum with $j = \alpha$ and $c_\alpha=c$, by (\ref{defF2}) means in fact that

\begin{equation}\label{defF22}
F(X) = \frac{c} {(1- X)^{\alpha}}
\end{equation}

Thus, by Lemma \ref{mainlemma} with $d=\alpha$, we deduce 
that $M_{\Omega, \mu}=\B^{n+1}$ and $\alpha=n+2$, concluding the proof of (1)
(notice that in this case $\mu=1$ and $\gamma=n+1$
in accordance with \eqref{relationgmul}).
The proof of (2) is obtained in a similar way and it is omitted.
\end{proof}

\end{document}